\documentclass[12pt]{amsart}
\usepackage[margin=1in]{geometry}

\usepackage{amsmath,amssymb,amsthm, amscd}
\usepackage{mathrsfs}
\usepackage[all]{xy}

\usepackage[colorinlistoftodos]{todonotes}

\newcommand\Z{\mathbb{Z}}

\newcommand\Q{\mathbb{Q}}

\newcommand\N{\mathbb{N}}

\newcommand{\PP}{\mathbb{P}}      
\newcommand{\CC}{\mathbb{C}}      
\newcommand{\msS}{\mathscr{S}}
\newcommand{\scW}{\mathscr{W}}
\newcommand{\scU}{\mathscr{U}}
\newcommand{\scX}{\mathscr{X}}

\newcommand{\mfp}{\mathfrak{p}}

\def\p{{\mathfrak{p}}}

\newenvironment{Gal}{{\rm Gal\,}}{}

\newenvironment{Spec}{{\rm Spec}}{}

\newtheorem{theorem}{Theorem}
\newtheorem{definition}[theorem]{Definition}
\newtheorem{lemma}[theorem]{Lemma}

\newtheorem{proposition-definition}[theorem]{Proposition-Definition}

\newtheorem{conjecture}[theorem]{Conjecture}

\theoremstyle{definition}
\newtheorem{example}[theorem]{Example}

\theoremstyle{remark}
\newtheorem*{remark}{Remark}

\title[Ramification and PCF morphisms]{Finite ramification for preimage fields of postcritically finite morphisms}

\begin{document}

\author[Bridy et al.]{Andrew Bridy, Patrick Ingram, Rafe Jones, Jamie Juul, Alon Levy, Michelle Manes, Simon Rubinstein-Salzedo, and Joseph~H.~Silverman}

\address{University of Rochester,  Rochester, New York}
\email{abridy@ur.rochester.edu}
\address{Colorado State University, Fort Collins, Colorado}
\email{pingram@rams.colostate.edu\footnote{Corresponding author.}}
\address{Carleton College, Northfield, Minnesota}
\email{rfjones@carleton.edu}
\address{Amherst College, Amherst, Massachusetts}
\email{jjuul@amherst.edu}
\address{KTH Royal Institute of Technology,  Stockholm, Sweden}
\email{alonlevy@kth.se}
\address{University of Hawaii, Honolulu, Hawaii}
\email{mmanes@math.hawaii.edu}
\address{Stanford University, Stanford, CA}
\email{simonr@stanford.edu}
\address{Brown University, Providence, Rhode Island}
\email{jhs@math.brown.edu}

\thanks{%
The authors would like to thank AIM and the organizers of the March
2014 AIM workshop on Postcritically finite maps in complex and
arithmetic dynamics at which this research was started.  
Bridy's research partially supported by NSF Grant \#EMSW21-RTG.  
Ingram's research partially supported by Simons Collaboration Grant \#283120.
Juul's research partially supported by DMS-1200749.  
Levy's research partially supported by the G\"oran Gustafsson Foundation.
Manes's research partially supported by NSF DMS \#1102858.
Silverman's research partially supported by Simons Collaboration Grant \#241309.
}

\begin{abstract} Given a finite endomorphism $\varphi$ of a variety $X$ defined over the field of fractions $K$ of a Dedekind domain, we study the extension $K(\varphi^{-\infty}(\alpha)) : = \bigcup_{n \geq 1} K(\varphi^{-n}(\alpha))$ generated by the preimages of $\alpha$ under all iterates of $\varphi$. In particular when $\varphi$ is post-critically finite, i.e. there exists a non-empty, Zariski-open $W \subseteq X$ such that $\varphi^{-1}(W) \subseteq W$ and $\varphi : W \to X$ is \'etale, we prove that $K(\varphi^{-\infty}(\alpha))$ is ramified over only finitely many primes of $K$. This provides a large supply of infinite extensions with restricted ramification, and generalizes results of Aitken-Hajir-Maire \cite{aitken} in the case $X = \mathbb{A}^1$ and Cullinan-Hajir, Jones-Manes \cite{cullinan, jones-manes} in the case $X = \mathbb{P}^1$. Moreover, we conjecture that this finite ramification condition characterizes post-critically finite morphisms, and we give an entirely new result showing this for $X = \mathbb{P}^1$. The proof relies on Faltings' theorem and a local argument. 
\end{abstract}

\maketitle

\section{Introduction}

For a number field $K$ and a finite set $S$ of places of $K$, the arithmetic fundamental group $\Gal(K_S/K)$ is the Galois group of the maximal algebraic extension of $K$ unramified outside of $S$. These groups are objects of great number-theoretic interest, and well-known quotients of them arise through algebraic geometry, for example in the form of linear representations attached to abelian varieties. Recently another approach to constructing such quotients has arisen, one that relies on arithmetic dynamical systems. It was first observed in \cite{aitken} that adjoining iterated pre-images of $\alpha \in K$ under a polynomial $f \in K[x]$ gave rise to infinite extensions unramified outside a finite set of primes, \textit{provided} that $f$ is post-critically finite, i.e., the forward orbit of each critical point of $f$ is finite. In \cite{cullinan} and \cite[Theorem 3.2]{jones-manes} this approach was extended to the case where $f$ is a rational function. In this paper, we extend this result even further, to the case where $f$ is allowed to be any finite endomorphism of a smooth, irreducible variety, and $K$ is allowed to be the fraction field of any Dedekind domain. We conjecture, however, that this finite ramification phenomenon for arithmetic dynamical systems is limited to those that are post-critically finite, and we prove this in the case where $f$ is a rational function.

We begin by extending the 
definition of post-critical finiteness of a rational function to the case of endomorphisms of more general varieties. 
\begin{definition} \label{pcfdef}
Let $X$ be an irreducible, smooth variety, and let $\varphi:X\to X$ be a finite morphism. We say that $\varphi$ is \textbf{post-critically finite (PCF)} if there is a non-empty Zariski-open $W\subseteq X$ such that $\varphi^{-1}(W)\subseteq W$, and such that $\varphi:W\to X$ is \'{e}tale.
\end{definition}

Note that a PCF map must be separable, since an inseparable morphism is nowhere \'{e}tale. In the case $X=\PP^N$, over the complex numbers, it is easy to see that Definition~\ref{pcfdef} is equivalent to the usual one, namely that if $C_\varphi\subseteq \PP^N_\CC$ is the critical locus of $\varphi$, then the sequence $\varphi^n(C_\varphi)$ is supported on a finite collection of irreducible subvarieties of $\PP^N_\CC$ as $n\to\infty$. Indeed, this latter condition holds if and only if  the union
\begin{equation} \label{union}
\bigcup_{n > 0} \varphi^n(C_\varphi)\subseteq\PP^N_\CC
\end{equation}
 is an algebraic subvariety of $\PP^N_\CC$, and thus the condition in Definition~\ref{pcfdef} is satisfied, where we take $W$ to be the complement of the union in \eqref{union}. 

Although the definition of PCF maps is purely geometric, PCF maps have special arithmetic properties besides the finite ramification of preimage fields explored in this article. See~\cite{favre-gauthier, pink} for some recent examples.

In this paper, we consider the primes that ramify in extensions obtained by taking backward images of $K$-points of $X$, where we typically take $K$ to be a global field.

\begin{definition}\label{k-infty}Let $K$ be any field, and let $\varphi: X \to X$ be defined over $K$. If $\alpha \in X(K)$ and $n$ is a positive integer, we set $K_{n, \varphi}(\alpha) = K(\varphi^{-n}(\alpha))$, and $K_{\infty, \varphi}(\alpha) = \bigcup_{n > 0}K(\varphi^{-n}(\alpha))$. If $K$ is a global field, we set $S_{n, \varphi}(\alpha)$ to be the set of primes of $K$ that ramify in $K_{n, \varphi}(\alpha)$ for $n \in \N \cup \{\infty\}$. When $\varphi$ is clear from context, we suppress it and write $K_n(\alpha)$ and $S_n(\alpha)$.\end{definition} 

Our first theorem says that if a map is PCF, then the field generated by the full iterated backward orbit of an algebraic point is ramified at only finitely many primes.

\begin{theorem}\label{th:finitely ramified}
Let $X$ be a smooth, irreducible, projective variety defined over the field of fractions $K$ of a Dedekind domain $R$, and suppose that $\varphi: X \to X$ is a PCF morphism defined over $K$. Then there exists a non-empty Zariski-open $W\subseteq X$ such that the following holds: for any $\alpha \in W(K)$, $S_\infty(\alpha)$ is a finite set.
\end{theorem}

\begin{remark}An equivalent way of stating Theorem~\ref{th:finitely ramified} is that the ramification of the associated so-called arboreal representation (see \cite{boston-jones1, boston-jones2, jones-manes}) has finite support.\end{remark}

Theorem~\ref{th:finitely ramified} is in the same spirit as Beckmann's theorem \cite[Theorem 1.2]{beckmann1} on ramified primes in specializations of branched covers of $\PP^1$. It is also interesting to consider the dependence of $S_\infty(\alpha)$ on the point $\alpha$. The existence of primes that ramify independently of the choice of $\alpha$ is related to ramification in the field of moduli of $\phi$. See~\cite{beckmann2},~\cite{fulton}, and~\cite{zannier1} for further discussions. For certain maps it is possible to show that there are primes (for example primes of bad reduction) that lie in $\bigcap_{\alpha\in W(K)}S_\infty(\alpha)$. It is not true in general that $\bigcup_{\alpha\in W(K)}S_\infty(\alpha)$ is a finite set (see Examples~\ref{P2ex} and~\ref{Tch-ex}). 

The subvariety $W\subseteq X$ in Theorem~\ref{th:finitely ramified} can, at worst, be taken to be the $W$ from Definition~\ref{pcfdef} (i.e., the complement of the postcritical set), as shown in the proof of Theorem~\ref{th:finitely ramified}. In the case $X=\PP^1$, we show in Theorem \ref{P1} that we may in fact take $W = X$, thus reproducing in this case the existing results in the literature (\cite{aitken, cullinan, jones-manes}). Levy and Tucker~\cite{levy-tucker} have announced a proof of a related conjecture about the behavior of PCF morphisms under restriction that would imply that Theorem~\ref{th:finitely ramified} is true with $W = X$.

We conjecture that the finite ramification of preimage fields given by Theorem~\ref{th:finitely ramified} in fact characterizes PCF morphisms:

\begin{conjecture}\label{converseconj}
Let $X$ be a smooth, irreducible, projective variety defined over the field of fractions $K$ of a Dedekind domain $R$. Suppose that $\varphi: X \to X$ is a non-PCF morphism defined over $K$, $\alpha \in X(K)$, and $\bigcup_{n > 0}\varphi^{-n}(\alpha)$ is an infinite set. Then $S_\infty(\alpha)$ is an infinite set.
\end{conjecture}
If $\bigcup\varphi^{-n}(\alpha)$ is not Zariski-dense then we say that $\alpha$ is an \emph{exceptional point} for $\varphi$. For $X = \PP^1$ this means that $\bigcup\varphi^{-n}(\alpha)$ is a finite set, which occurs only when, after coordinate change, $\alpha = \infty$ and $\varphi$ is a polynomial, or when $\alpha \in \{0, \infty\}$ and $\varphi(z) = z^d$. If $\dim X > 1$ and $\alpha$ is exceptional for $\varphi$ but $\bigcup_{n > 0}\varphi^{-n}(\alpha)$ is infinite, then the Zariski-closure of $\bigcup_{n > 0}\varphi^{-n}(\alpha)$ is a positive-dimension subvariety $Y$ that is fixed under $\varphi$, and $\alpha$ is not exceptional under $\varphi|_Y$.

Conjecture~\ref{converseconj} is known only in the case where $K$ is a function field over a number field $k$, $\varphi$ is a univariate polynomial defined over $k$, and $\alpha$ is transcendental over $k$ \cite[Theorem 1.1]{aitken}. We prove the following:

\begin{theorem} \label{P1converse}
Conjecture~\ref{converseconj} holds when $K$ is a number field and $X = \PP^1$. 
\end{theorem}


We also raise a second conjecture here: in Theorem~\ref{th:finitely ramified}, we get an infinite field extension of $K$ that is ramified over finitely many primes of $K$. It is a longstanding problem to study extensions of number fields with restricted \textit{tame} ramification, and it would be interesting to have a dynamical source of such extensions. We conjecture, however, that this is not possible: 

\begin{conjecture}\label{wild}
In the situation of Theorem~\ref{th:finitely ramified}, there always exists at least one prime of $K$ in $S_\infty(\alpha)$ where the ramification is wild.
\end{conjecture}


\section{Proof of Theorem~\ref{th:finitely ramified}}\label{Proof}

Before giving the proof of Theorem~\ref{th:finitely ramified}, we state and prove an appropriate version of the Chevalley-Weil Theorem.

\begin{theorem}[Chevalley-Weil]
Let $\varphi: \mathscr{X} \to \mathscr{Y}$ be a finite \'{e}tale map of schemes defined over $A$, and let $P\in \mathscr{Y}(A)$. Then there is a finite, unramified extension $B/A$ of rings and a point $Q\in\mathscr{X}(B)$ such that $f(Q)=P$.
\end{theorem}

\begin{proof}
First, let $\mathscr{Z}=\mathscr{X}\times_{\mathscr{Y}}\Spec(A)$ be the fiber product relative to the maps $\varphi:\mathscr{X}\to\mathscr{Y}$ and $P:\Spec(A)\to\mathscr{Y}$. Since $P$ and $\varphi$ are finite, so is the map $\mathscr{Z}\to\Spec(A)$, and hence $\mathscr{Z}=\Spec(B)$ for some finite $A$-module $B$ \cite[p.~91]{hartshorne}. By definition, the fiber product gives a map $\Spec(B)\to \scX$, i.e., a point $Q\in \scX(B)$ with $P=\varphi(Q)$. Here, we abuse notation and use $P$ to refer to the canonical map $\Spec(B)\to\mathscr{Y}$ induced by composing the canonical map $\Spec(B)\to\Spec(A)$ with $P:\Spec(A)\to\mathscr{Y}$. Finally, the fact that $B/A$ is unramified follows from the fiber product of \'{e}tale morphisms being \'{e}tale.
\end{proof}

\begin{proof}[Proof of Theorem \ref{th:finitely ramified}]
Let $R$ be a Dedekind domain, with fraction field $K$, let $X/K$ be an irreducible, smooth, projective variety, let $\varphi:X\to X$ be a PCF morphism defined over $K$, let $W\subseteq X$ be a set as in Definition~\ref{pcfdef}, and let $\alpha\in W(K)$. We now describe a finite set of primes which will contain all primes ramifying in any of the extensions $K_n/K$. Specifically, let $S\subseteq \Spec(R)$ be the set of places $\mfp$ for which $X$ has bad reduction modulo $\mfp$, $\varphi$ has bad reduction modulo $\mfp$, the ramification subscheme of $\varphi$ has a fibral component modulo $\mfp$, or $\alpha$ collides with the complement of $W$ modulo $\mfp$. In other words, if $R_S$ is the localization of $R$ at $S$, then $X$ is the generic fiber of a scheme $\scX\to\Spec(R_S)$, $W$ is the generic fiber of a scheme $\scW\to\Spec(R_S)$, $\alpha$ extends to a morphism $\widetilde{\alpha}:\Spec(R_S)\to \scW$, and $\varphi$ extends to an \'{e}tale morphism $\varphi:\scW\to\scX$ over $R_S$. Now, note that if $\scU=\varphi^{-n}(\scW)$, then since the inclusion map $\scU\to\scW$ is \'{e}tale, so is the map $\varphi^n:\scU\to\scW$.
Theorem~\ref{th:finitely ramified} is the application of the Chevalley-Weil Theorem to the point $\widetilde{\alpha}\in \scW(R_S)$ and the map $\varphi^n:\scU\to\scW$.
\end{proof}

\begin{remark}The ramification subscheme of $\varphi$ mod $\p$ has a fibral component if and only if the reduction of $\varphi$ mod $\p$ is inseparable.\end{remark}


\begin{theorem} \label{P1}If $X = \PP^1$ then Theorem~\ref{th:finitely ramified} holds with $W = X$.\end{theorem}

\begin{proof}The proof of Theorem~\ref{th:finitely ramified} explicitly constructs $W$ to be $X$ minus the postcritical set of $\varphi$, which when $X = \PP^1$ and $\varphi$ is PCF is a finite set of points. The critical locus $C_\varphi$ consists of a finite set of points, and there exists some constant $m$ such that if $\beta \in C_\varphi$ then $\alpha = \varphi^i(\beta)$ for some $i \leq m$ or $\alpha$ is not in the forward orbit of $\beta$.

Suppose first that $\alpha$ is not a periodic point of $\varphi$. Then $\varphi^{-m}(\alpha)$ consists of a finite set of points, none of which is of the form $\varphi^i(\beta)$ for any $\beta \in C_\varphi$. Thus, $K_\infty(\alpha)/K_m(\alpha)$ is finitely ramified. $K_m(\alpha)/K$ is finitely ramified since it is a finite extension, and therefore $K_\infty(\alpha)/K$ is finitely ramified, as required.

If $\alpha$ is periodic, then we first observe that $K_n(\alpha) \subseteq K_m(\alpha)$ if $n < m$, and therefore $$K_{\infty, \varphi}(\alpha) = \bigcup_{i = 1}^\infty K_{i, \varphi}(\alpha) = \bigcup_{i = 1}^\infty K_{ni, \varphi}(\alpha) = K_{\infty, \varphi^n}(\alpha)$$ for every $n \geq 1$. Since $K_{\infty, \varphi}(\alpha) = K_{\infty, \varphi^n}(\alpha)$, we may freely replace $\varphi$ with an iterate, and in particular we may assume that $\alpha$ is fixed. But now $K_n(\alpha)$ is the compositum of $K_{n-1}(\beta)$ over all $\beta \in \varphi^{-1}(\alpha) \setminus \{\alpha\}$. Since $\beta$ is not periodic, $K_\infty(\beta)$ is finitely ramified, and this implies $K_\infty(\alpha)$ is finitely ramified.\end{proof}

\section{Examples}

The proof of Theorem~\ref{th:finitely ramified} is sufficiently explicit that a set of primes outside of which the preimage extensions are unramified can often be presented with relative ease.
\begin{example} \label{P2ex}
Dupont~\cite{dupont} produces an example of a quadratic map $\varphi:\PP^2\to\PP^2$ which is post-critically finite, namely
\[\varphi[x:y:z]=\left[(x-y+z)^2:(x+y-z)^2:(-x+y+z)^2\right].\]
One checks that the post-critical locus of this morphism is precisely that defined by
\[xyz(x-y)(y-z)(z-x)=0.\]
Thus, in the proof of Theorem~\ref{th:finitely ramified}, we may take $W\subseteq\PP^2$ to be the complement of this divisor.

Given a number field $K$, and a point $P\in W(K)$, let $S$ be any set of places of $K$ containing all places above 2, above the coordinates of $P$, and above their pairwise differences. Then if $\mathcal{O}_{K, S}$ is the set of $S$-integers of $K$, the morphism $\varphi:W\to W$ is the generic fiber of an \'{e}tale map $\varphi:\mathcal{W}\to \mathcal{W}$ of $\mathcal{O}_{K, S}$-schemes, and $P$ extends to a point $P\in \mathcal{W}(\mathcal{O}_{K, S})$. It follows from the proof of Theorem~\ref{th:finitely ramified} that the extensions $K(\varphi^{-n}(P))$ are all unramified above primes outside of $S$.

Incidentally, we note that, given a set $S$, there are only finitely many points $P\in W(K)$ for which this set of places suffices, which in this case follows from the finiteness of solutions to $S$-unit equations. More generally, if $W\subseteq\PP^2$ is an affine open set containing infinitely many $S$-integral points, this places substantial restrictions on the complement $\PP^2\setminus W$ (in our case, the postcritical set); see, e.g., work of Corvaja and Zannier~\cite{cz}.

Observe that the periodic postcritical components are $x = y$, $y = z$, and $z = x$, which $\varphi$ permutes in a $3$-cycle. Let us look at how $\varphi^3$ acts on these components. Since $\varphi$ has a $\Z/3\Z$-automorphism by the coordinate change $x \mapsto y \mapsto z \mapsto x$, it suffices to look at the action of $\varphi^3$ on just one of these components. Let us look at the component $x = y$, and let us dehomogenize by setting $z = 1$. It is a straight forward computation that $$\varphi^3(x) = \frac{(4x^2 - 4x - 1)^4}{(16x^4 - 32x^3 + 40x^2 - 24x + 1)^2}$$

There are $14$ critical points, counted with multiplicity: the two roots of $4x^2 - 4x - 1$ are quadruple zeros and triple critical points, the roots of $16x^4 - 32x^3 + 40x^2 - 24x + 1$ are double poles and simple critical points, $\infty$ is a simple critical point (the top two terms of both numerator and denominator are the same), and after factoring those out we obtain $1/2$ as a triple critical point. $0, 1/2, \infty$ all map to the fixed point $1$. More specifically, the critical values are $0, 1, \infty$, which are precisely the intersections of the other postcritical components with $x = y$. As expected based on the in-progress work in \cite{levy-tucker}, the restriction of $\varphi^3$ to $x = y$ is PCF. If we pull back a point on this line, then again we obtain finite ramification.
\end{example}

\begin{example} \label{Tch-ex}
As another example, consider the generalized Tchebyshev map $\varphi:\PP^2\to\PP^2$ given by
\[\varphi[x:y:z]=[x^2-2yz:y^2-2xz:z^2],\]
whose postcritical set consists of the quintic curve
\[z(x^2y^2-4x^3z-4y^3z+18xyz^2-27z^4)=0.\]
If $W$ is the complement of this curve, and $P\in W(K)$, then there is a finite set $S$ of primes such that over $\mathcal{O}_{K, S}$, $\varphi$ extends to an \'etale map $\varphi:\mathcal{W}\to\mathcal{W}$, and $P$ extends to a point $\mathcal{W}(\mathcal{O}_{K, S})$. The extensions $K(\varphi^{-n}(P))$ are again unramified above primes outside of $S$. It follows again from the result of Corvaja and Zannier~\cite{cz} that a given set $S$ will work for only finitely many points.

In this example, there is in fact a simpler way of accessing this fact (noting that computing the set of primes modulo which a given point intersects the above quintic is perhaps not trivial). The Tchebyshev example above fits into a commutative diagram
\[\begin{CD}
\mathbb{G}_\mathrm{m}^2 @>{Q\mapsto Q^2}>> \mathbb{G}_\mathrm{m}^2\\
@V{\pi}VV @VV{\pi}V\\
\PP^2 @>>{\varphi}> \PP^2
\end{CD}\]
for $\pi(x, y)=\left(x+y+\frac{1}{xy}, \frac{1}{x}+\frac{1}{y}+xy\right)$, which is ramified exactly where $x=y$, $x=y^{-2}$, or $y=x^{-2}$. Thus if $P\in \PP^2(K)$ satisfies $P=\pi(Q)$, and $S$ is a set of places of $K$ for which $Q$ is an $S$-unit, then the extension $K(\varphi^{-n}(P))$, contained in $K(Q^{1/2^n})$, will be unramified outside of $S$.
\end{example}

\section{A characterization of PCF morphisms on $\mathbb{P}^1$}

Let us now turn to the proof of Theorem \ref{P1converse}; throughout this section, unless otherwise noted, we assume that $K$ is a number field. The key ingredient of the proof of Theorem~\ref{th:finitely ramified} is that a given $\alpha \in X(K)$ will only collide with a proper closed algebraic subset of $X$ modulo finitely many primes. If $f: X \to X$ is postcritically infinite, then we expect $\alpha$ to collide with the postcritical set modulo infinitely many primes. However, such collision modulo a prime alone is insufficient to guarantee ramification of pre-image fields at that prime.

\begin{example}Let $p$ be a prime number, $f(z) = z(z-p)$, and $\alpha = 0$. The unique finite critical point is $p/2$, with critical value $-p^2/2$. Although $-p^2/2 \equiv 0 \mod p$, we have $\Q(f^{-1}(0)) = \Q$, and so there is no ramification in $\Q(f^{-1}(0))$.\end{example}

It is not hard to prove that, if $X = \PP^1$, then for every infinite postcritical orbit and every $\alpha$ that does not lie on it, there exist infinitely many primes modulo which $\alpha$ does lie in the orbit. Indeed, by results of Silverman~\cite{silverman}, this is true of every infinite forward orbit. Our proof of Theorem \ref{P1converse} requires the following stronger version of this:

\begin{lemma}\label{divisors}Let $\varphi(z): \PP^1_K \to \PP^1_K$, let $e \geq 2$ be an integer, suppose that $0$ is not a postcritical point under $\varphi$,  suppose that $a\in\mathbb{P}^1(K)$ is not preperiodic, and let $S$ be any finite set of places of $K$. Then there exist $\mathfrak{p}\not\in S$ and $n\geq 0$ such that $v_{\mathfrak{p}}(\varphi^n(a))$ is positive and not divisible by~$e$. 
\end{lemma}

\begin{proof}
Before beginning the proof, we note that it is enough to establish the claim for some iterate of $\varphi$. If $0$ is not post-critical for $\varphi$, then it is not post-critical for $\varphi^k$ either, for any $k\geq 1$. Since every iterate of $\varphi^k$ is also an iterate of $\varphi$, the result for $\varphi$ now follows from that for $\varphi^k$. We may replace $\varphi$ with an iterate, then, to assume that $\deg(\varphi)\geq 5$.
Note also that the numerator of $\varphi$ has no repeated roots, since such a root $\beta$ would satisfy $\varphi(\beta)=\varphi'(\beta)=0$, which contradicts our hypothesis that $0$ is not post-critical for $\varphi$. Since $\infty$ is also not mapped by $\varphi$ to 0 with multiplicity, we may assume that the numerator of $\varphi$ (written in lowest terms as a quotient of two polynomials) has at least 4 distinct simple roots.

Now fix $\varphi$, $e$, and $a$ as in the statement and suppose, toward a contradiction, that there is some finite set $S$ of places of $K$ such that for all $\mathfrak{p}\not\in S$ and all $n\geq 0$, if $v_{\mathfrak{p}}(\varphi^n(a))>0$ then $e\mid v_{\mathfrak{p}}(\varphi^n(a))$. We will enlarge our set $S$ below, which clearly does not disrupt this property.

Write $\varphi^n(a_0)=\alpha_n/\beta_n$ with $\alpha_n, \beta_n\in \mathcal{O}_K$ chosen so that $\alpha_n\mathcal{O}_K+\beta_n\mathcal{O}_K$ all divide some fixed ideal $I$ (we can do this by the finiteness of the class group). Write $\varphi(x/y)=F(x, y)/G(x, y)$ with integral coefficients, and enlarge the set $S$ of places enough that $\operatorname{Res}(F, G)\in \mathcal{O}_{K, S}^\times$ and such that $\mathcal{O}_{K, S}$ is a PID in which $I$ generates the trivial ideal.

For a given $n$, write $\varphi^n(a_0)\mathcal{O}_{K, S}=\mathfrak{a}/\mathfrak{b}$, where $\mathfrak{a}$ and $\mathfrak{b}$ are coprime ideals in $\mathcal{O}_{K, S}$. Note that, by our construction of $S$, we have
\[F(\alpha_{n-1}, \beta_{n-1})\mathcal{O}_{K, S}=\mathfrak{a}.\]
We have assumed that $e\mid v_{\mathfrak{p}}(\varphi^n(a_0))$ for every $n\geq 0$ and every $\mathfrak{p}\not\in S$ with $v_{\mathfrak{p}}(\varphi^n(a_0))>0$, and so (since $\mathcal{O}_{K, S}$ is a PID) we have
\[F(\alpha_{n-1}, \beta_{n-1})=sy^e,\]
for some $S$-unit $s$. Choosing a finite set $\msS$ of coset representatives of $\mathcal{O}_{K, S}^\times/(\mathcal{O}_{K, S}^\times)^e$, we in fact have that for each $n\geq 0$,
\begin{equation}\label{eq:dgeq}
s^{-1}F(\alpha_{n-1}, \beta_{n-1})=y^e,
\end{equation}
for some $y\in K$ and some $s\in \msS$. Since $\msS$ is finite, there is one particular $s\in \msS$ such that~\eqref{eq:dgeq} has a solution with $y\in K$ for all $n$ in some infinite set $Z$. Since $\alpha_n \mathcal{O}_K+\beta_n \mathcal{O}_K$ all divide some fixed ideal $I$, and since we have that $F(x, 1)$ has four distinct simple roots, we may apply a result of Darmon and Granville~\cite[Theorem~1$'$]{darmon-granville} to conclude that there are only finitely many distinct values $\varphi^n(a_0)=\alpha_n/\beta_n$ with $n\in Z$. But $Z$ is infinite, and this would mean that $a_0$ is preperiodic for $\varphi$. The lemma follows from this contradiction.
\end{proof}

From Lemma \ref{divisors}, we now prove Theorem \ref{P1converse}.

\begin{proof}[Proof of Theorem~5]
We first make some reductions of the theorem. Note that the compositum of a finite extension of $K$ with a finitely-ramified extension will again be finitely ramified, so it suffices to prove the theorem after adjoining some algebraic values to $K$. In particular, we will assume that $\varphi$ has a critical point $\zeta\in\mathbb{P}^1(K)$ whose orbit is infinite, and we denote by $e\geq 2$ the local degree of $\varphi$ at this point.

We write $K_n(\alpha)=K(\varphi^{-n}(\alpha))$. If $\varphi^k(\beta)=\alpha$, then $K_n(\beta)\subseteq K_{n+k}(\alpha)$, so it suffices to prove the theorem after replacing $\alpha$ with any of its preimages. Since $\alpha$ is not exceptional, it has infinitely many preimages, only finitely many of which can be postcritical. Without loss of generality, then, we will suppose that $\alpha$ is not post-critical.
Finally, making a change of coordinates if necessary, we assume for convenience that $\alpha=0$ and that $\varphi(\infty)=\infty$.
When we write \[\varphi([x:y])=\frac{F(x, y)}{G(x, y)},\]
with $F$ and $G$ coprime homogeneous forms, this last assumption gives that $G(x, y)$ is divisible by $y$.

Toward a contradiction, suppose that there exists a finite set $S$ of places of $K$ such that the fields $K_n(0)/K$ are all unramified outside of $S$. We will replace $S$ with a larger finite set such that the coefficients of $F$ and $G$ are $S$-integral; $\operatorname{Res}(F, G)$ is an $S$-unit; $\zeta$ and $\varphi(\zeta)$ are $S$-integral; and the constant terms of $G(x+\zeta, 1)$ and $F(1, x)$, and the quantity $\varphi^{(e)}(\zeta)/e!$, are $S$-units, where $\varphi^{(e)}$ denotes the $e$th derivative of $\varphi$ evaluated at $\zeta$ (which is necessarily non-zero). Note that it follows from this that $\varphi(z+\zeta)$ admits a power series expansion with $S$-integral coefficients of the form
\[\varphi(z+\zeta)=\varphi(\zeta)+\frac{\varphi^{(e)}(\zeta)}{e!}z^e+O\left(z^{e+1}\right),\]
with the coefficient of $z^e$ an $S$-unit.
Also note that if we set
\[F_{n+1}(x, y)=F(F_n(x, y), G_n(x, y))\]
and 
\[G_{n+1}(x, y)=G(F_n(x, y), G_n(x, y))\]
then $\operatorname{Res}(F_n, G_n)$ is always an $S$-unit, and $F_n$ and $G_n$ always have $S$-integral coefficients.

By Lemma~12 there exists a prime $\mathfrak{p}\not\in S$ such that $v_{\mathfrak{p}}(\varphi^n(\zeta))$ is positive and not divisible by $e$. We have
\[0<v_{\mathfrak{p}}(\varphi^n(\zeta))=v_{\mathfrak{p}}(F_{n-1}(\varphi(\zeta),1))-v_{\mathfrak{p}}(G_{n-1}(\varphi(\zeta),1)).\]
But $F_{n-1}(\varphi(\zeta), 1)$ and $G_{n-1}(\varphi(\zeta), 1)$ are $S$-integers with no common factor, and so in fact
\[v_{\mathfrak{p}}(\varphi^n(\zeta))=v_{\mathfrak{p}}(F_{n-1}(\varphi(\zeta),1)).\]

Let $\mathfrak{P}$ be a prime of $K_{n}(0)$ dividing $\mathfrak{p}$. Note that, since $\mathfrak{p}$ does not ramify, we have $v_{\mathfrak{P}}(x)=v_{\mathfrak{p}}(x)$ for all $x\in K$. 

If we write $F_n(x, 1)=s_nx^{d^n}+\cdots$ for each $n$, then a simple induction using the fact that $y$ divides $G(x, y)$ shows that $s_n=s_1^{1+d+d^2+\cdots+d^{n-1}}$. In particular, this value is an $S$-unit for all $n$ (given our assumption that it is for $n=1$).

Factoring in $K_{n-1}(0)$, we have
\[F_{n-1}(x, y)=s_{n-1}(x-\delta_1 y)\cdots (x-\delta_{d^{n-1}} y),\]
where $\delta_1, ..., \delta_{d^{n-1}}$ are necessarily $S$-integral. 
If $e\mid v_{\mathfrak{P}}(\varphi(\zeta)-\delta_i)$ for each $i$, then the same would be true for \[v_{\mathfrak{P}}(F_{n-1}(\varphi(\zeta), 1))=\sum_{i=1}^rv_{\mathfrak{P}}(\varphi(\zeta)-\delta_i),\] which we know not to be the case. So we have some $i$ with $e\nmid v_{\mathfrak{P}}(\varphi(\zeta)-\delta_i)>0$.

Now consider the function $\psi(z)=\varphi(z+\zeta)-\delta_i$.
The roots of $\psi(z)$ are all $K_n(0)$-rational, since each is of the form $\gamma-\zeta$, with $\gamma$ an $n$th preimage of $0$ (more specifically, a first preimage of $\delta_i$).
By hypothesis, this function admits a power series expansion which converges on the $\mathfrak{P}$-adic open unit disk. On the other hand, we know that this power series has the form
\[\psi(z)=\left(\varphi(\zeta)-\delta_i\right)+\frac{\varphi^{(e)}(\zeta)}{e!}z^e+O(z^{e+1}),\]
where the remaining terms all have $\mathfrak{P}$-integral coefficients. The Newton Polygon for $\psi$, then, has an initial segment of length $e$ and slope $-v_{\mathfrak{P}}(\varphi(\zeta)-\delta_i)/e$, followed by an infinite segment of slope 0. Consequently, one of the roots of $\psi$ has valuation $v_{\mathfrak{P}}(\varphi(\zeta)-\delta_i)/e$ (indeed, $e$ of them do), but such roots are $K_n(0)$-rational, and thus have integral $\mathfrak{P}$-valuation. We conclude that $v_{\mathfrak{P}}(\varphi(\zeta)-\delta_i)$ is divisible by $e$. This is a contradiction.
\end{proof}

\begin{remark}In higher dimension, if $\alpha$ is exceptional, then its backward image is contained in a positive-codimension closed subvariety $Y \subsetneq X$. In line with Conjecture~\ref{converseconj}, we conjecture that whenever the restriction of $\varphi$ to $Y$ is postcritically infinite, we again get ramification at infinitely many primes. The reason for the restriction to non-exceptional points in the statement of Conjecture~\ref{converseconj} is that even when $\varphi$ is postcritically infinite, its restriction to $Y$ may be PCF.\end{remark}


If Lemma~\ref{divisors} is true for a higher-dimensional variety $X$, it is not too difficult to generalize Theorem~\ref{P1converse}. The proof of Theorem~\ref{P1converse} can be reinterpreted in Newton polygon language, which readily generalizes to higher dimension; for a survey, see~\cite{rabinoff}.

The other step in the proof requires showing that a given $\alpha$ collides with one postcritical component at a time; this is why $S$ contains the set of primes where two critical points of $\varphi$ collide. To generalize this to higher dimension, first, we note that at all but finitely many primes --- indeed, generically, for all primes --- no two components of the critical locus, a finite union of irreducible hypersurfaces, will have the same reduction.

Now, we need to deal with the case in which $\alpha$ collides with the intersection of $\varphi^{n}(C_1)$ and $\varphi^{n}(C_2)$ where $C_1$ and $C_2$ are two distinct critical components, mapping to their images with local degrees $e_1$ and $e_2$. In that case, $\alpha$ would have a preimage of multiplicity $e_1 e_2$, which is clearly divisible by $e_1$. More generally, the proper intersection of critical components of mapping degrees $e_1, \ldots, e_k$ maps to its preimage with degree $e_1\ldots e_k$; improper intersection will only happen modulo finitely many primes, which we can exclude. We can thus make $e_1$ the $e$ we use in the proof of Theorem~\ref{P1converse}.

The difficulty is in generalizing Lemma~\ref{divisors}. To obtain the same argument, we first must assume that the Bombieri-Lang conjecture is true. Although the conjecture merely says that general-type varieties have non-Zariski dense sets of rational points, rather than finite sets, in the particular example of $sy^e = f(\varphi(C))$, where $C$ is a hypersurface, $\varphi$ is its image, and $f(\varphi(C))$ is a normalized defining polynomial for $\varphi(C)$, Bombieri-Lang does in fact imply finiteness.

The problem is that we cannot place all the defining polynomials for the components of the postcritical locus into one equation. The components will have growing degrees, so they will come from Chow varieties of growing dimension. The converse, i.e.\ colliding a specified hypersurface with a Zariski-dense orbit of points, follows trivially from Bombieri-Lang, but for our purposes, we need to show that a specified non-postcritical $\alpha$ collides with the postcritical locus, to degrees not divisible by certain mapping degrees, modulo infinitely many primes. We cannot make any use of stronger conjectures, such as a uniform bound on the number of rational points of general-type varieties: such conjectures only apply for families with fixed discrete parameters, just as the uniform bound conjecture for rational points of general-type curves lets the number of points depend on the genus.

\bibliographystyle{halpha}
\bibliography{bibliography}

\end{document}